\documentclass[12pt, reqno]{amsart}
\usepackage{amssymb,a4}
\usepackage{amsmath,amsthm,amsfonts,amssymb}

\usepackage{amssymb,a4}
\usepackage{amsthm}

\def\N{\mathbb{N}}
\def\z{\mathbb{Z}}
\def\Z{\mathbb{Z}}
\def\c{\mathbb{C}}
\def\C{\mathbb{C}}
\def\dim{\hbox{dim}}

\def\ll{\lambda}

\newfont{\df}{eufm10}

\def\ll{\lambda}

\def\ot{\otimes}

\def\de{\delta}

\def\dim{\hbox{\rm dim}\,}

\def\Vir{\hbox{\rm Vir}}

\def\ot{\otimes}

\newtheorem{theo}{Theorem}[section]
\newtheorem{defi}[theo]{Definition}
\newtheorem{lemm}[theo]{Lemma}

\newtheorem{prop}[theo]{Proposition}
\numberwithin{equation}{section}

\begin{document}

\title[The Affine-Virasoro Algebra]{\bf Representations of the Affine-Virasoro \\ Algebra of type $A_1$}

\author[Y. Gao]{Yun Gao}
\address{Department of Mathematics and Statistics, York University, Toronto,
Canada M3J 1P3}
\email{ygao@yorku.ca}

\author[N. Hu]{Naihong Hu}
\address{Department of Mathematics, Shanghai Key Laboratory of Pure Mathematics and Mathematical Practice, East China Normal University, Shanghai, 200241, China}
\email{nhhu@math.ecnu.edu.cn}

\author[D. Liu]{Dong Liu$^\star$}
\address{Department of Mathematics, Huzhou University, Zhejiang Huzhou, 313000, China}
\email{liudong@zjhu.edu.cn}
\thanks{$^\star$ D. Liu, Corresponding author}


\date{}

\maketitle
\begin{abstract}
In this paper, we classify all irreducible weight
modules with finite-dimensional
weight spaces over the affine-Virasoro Lie algebra of type $A_1$. \vskip3mm
\end{abstract}

\section{Introduction}

It is well known that the affine Lie algebras and the Viasoro algebra have been widely used in many physics areas and mathematical branches, and the Virasoro algebra served as an outer-derivative
subalgebra plays a key role in representation theory of the affine Lie algebras. Their close relationship strongly suggests
that they should be considered simultaneously, i.e., as one algebraic
structure. Actually it has led to the definition of the so-called
affine-Virasoro algebra \cite{Ka1,Ku}, which is the semidirect
product of the Virasoro algebra and an affine Kac-Moody Lie algebra
with a common center. Affine-Virasoro algebras sometimes are much more connected to the conformal field theory.
For example, the even part of the $N=3$ superconformal algebra is
just the affine-Virasoro algebra of type $A_1$. Highest weight
representations and integrable representations of the affine-Virasoro
algebras have been studied in several papers (see
\cite{Ka1,EJ,Ku,JY,LH,LQ,W,XH}, etc.). All irreducible Harish-Chandra modules (weight
modules with finite-dimensional weight spaces) with nonzero central actions the affine-Virasoro
algebras were classified in \cite{B}. However, up to now, all irreducible uniform bounded modules over these algebras are not yet classified.

In this paper, we classify
all irreducible weight modules with finite-dimensional weight spaces over the affine-Virasoro Lie algebra of
type $A_1$. Throughout this
paper, $\z$, $\z^*$ and $\c$ denote the sets of integers, non-zero
integers and complex numbers, respectively. $U(L)$ denote the universal enveloping algebra of a Lie algebra $L$. All modules considered in this paper are nontrivial.   For any $\z$-graded space $G$, we also use notations $G_+, G_-, G_0$ and $G_{[p, q)}$ to denote the subspaces spanned by elements in $G$ of degree $k$ with $k>0$, $k<0$, $k=0$   and $p\le k<q$, respectively.

\section{Basics}
In this section, we shall introduce some notations of the Virasoro algebra and affine-Virasoro algebras.

\subsection{Virasoro algebra and twisted Heisenberg-Virasoro algebra}

By definition, the Virasoro algebra Vir:=$\c\{d_m, C\mid m\in\z\}$ with bracket:
\begin{equation}
[d_m, d_n]=(n-m)d_{m+n}+\de_{m+n,0}\frac{m^3-m}{12}C,\quad  [d_m, C]=0,
\end{equation} for all $m, n\in\z$.

Let $\c[t, t^{-1}]$
be the Laurent polynomials ring over $\c$, then $\mbox{Der}\,\c[t, t^{-1}]=\c\{t^{m+1}\frac d{dt}\mid m\in\z\}$
(also denote by Vect($S^1$), the Lie algebra of all vector fields on the circle).
$$\Vir=\widehat{\mbox{Der}\,\c[t, t^{-1}]}.$$

The twisted Heisenberg-Virasoro algebra $\mathcal H$ was first studied by
Arbarello et al in \cite{ADKP}, where a connection is established
between the second cohomology of certain moduli spaces of curves and
the second cohomology of the Lie algebra of differential operators
of order at most one.
By definition, $\mathcal H$ is the universal central extension of
the following Lie algebra $\mathcal D$, which is the Lie algebra of
differential operators order at most one.

\begin{defi} As a vector space over $\c$, the Lie algebra $\mathcal D$ has a
basis $\{d_n, Y_n\mid n \in\z\}$ with the following relations
\begin{eqnarray}
&&[d_m,d_n]=(n-m)d_{m+n},\\
&&[d_m,Y_n]=nY_{m+n}, \\
&&[Y_m,Y_n]=0,
\end{eqnarray} for all $m,\; n\in\z$.
\end{defi}

Clearly, the subalgebra $H=\c\{Y_m\mid m\in\z\}$ of $\mathcal D$ is centerless Heisenberg algebra and
$W=\c\{d_m\mid m\in\z\}$ is the Witt algebra (or centerless Virasoro algebra).

\subsection{Affine-Virasoro algebra}

\begin{defi} \label{aff-vir} Let $L$ be a finite-dimensional Lie algebra with a
non-degenerated invariant normalized symmetric bilinear form $(\, , )$, then the affine-Virasoro Lie algebra is the vector space
$$L_{av}=L\otimes \c[t, t^{-1}]\oplus\c C\oplus\bigoplus_{i\in\z}\c{d_i},$$
with Lie bracket:
\begin{eqnarray*}&&[\,x\otimes t^m, y\otimes t^n\,]=[\,x, y\,]\otimes t^{m+n}+m(x, y)\de_{m+n, 0}C,\\
&&[d_i, d_j]=(j-i)d_{i+j}+{1\over 12}(j^3-j)\de_{i+j, 0}C,\\
&&[d_i, x\ot t^m]=mx\ot t^{m+i}, \quad [C, L_{av}]=0,\end{eqnarray*}
where $x, \,y\in L$, $m,n, i, j\in\z$ (if $L$ has no such form, we set $(x, y)=0$ for all $x, y\in L$).
\end{defi}
\noindent{\bf Remark:}
If $L=\c{e}$ is one dimensional, then $L_{av}$ is just the
twisted Heisenberg-Virasoro algebra (one center element).

\vskip5pt

Now we only consider specially $L$ as the simple Lie algebra $\frak{sl}_2=\c\{e, f, h\}$. Then by Defintion \ref{aff-vir}, the corresponding affine-Virasoro algebra $\mathcal L:=L_{av}=\c\{e_i, f_i, h_i,d_i, C\mid i\in\z\}$,
with Lie bracket:
\begin{eqnarray*}&&[\,e_i, f_j\,]=h_{i+j}+i\de_{i+j, 0}C,\\
&&[\,h_i, e_j\,]=2e_{i+j}, \quad  [\,h_i, f_j\,]=-2f_{i+j},\\
&&[d_i, d_j]=(j-i)d_{i+j}+{1\over 12}(j^3-j)\de_{i+j, 0}C,\\
&&[d_i, h_j]=jh_{i+j}, \quad [h_i, h_j]=2i\de_{i+j, 0}C,\\
&&[d_i, e_j]=je_{i+j}, \quad [d_i, f_j]=jf_{i+j}, \quad [C, {\mathcal L}]=0,\end{eqnarray*}
where $i, j\in\z$.

\noindent{\bf Remark.} In fact, $\mathcal L$
is the even part of the $N=3$ superconformal algebra (\cite{CL}).

\vskip5pt
The Lie algebra $\mathcal L=\oplus_{i\in\z}{\mathcal L}_i$
is a $\z$-graded Lie algebra, where ${\mathcal L}_i=\c\{d_i, e_i, f_i, h_i, \delta_{0, i}C\}$.
The subalgebra ${\mathcal H}_X:=\c\{X_i, d_i, C\mid i\in\z\}$ for $X=e,f,h$ of $\mathcal L$
is isomorphic to the twisted
Heisenberg-Virasoro algebra ${\mathcal H}$ (the only difference is the center element).
Clearly $\frak h:=\c h_0+\c C+\c
d_0$ the Cartan subalgebra of ${\mathcal L}$.

\subsection{Harish-Chandra modules}

For any
$\mathcal L$-module $V$ and $\lambda, \mu\in \c$, set
$V_{\lambda, \mu}:=\bigl\{v\in V\bigm|d_0v=\lambda v, h_0v=\mu v\bigr\}$, which is
generally called the weight space of $V$ corresponding to the weight
$\lambda, \mu$.

An $\mathcal L$-module $V$ is called a weight module if $V$ is the
sum of all its weight spaces.

 A nontrivial irreducible weight ${\mathcal L}$-module $V$ is called of intermediate
 series if all its weight spaces are one-dimensional.
\par
A  weight ${\mathcal L}$-module $V$ is called a {highest} (resp.
{lowest) weight module} with {highest weight} (resp. {lowest
weight}) $\lambda, \mu\in \c$, if there exists a nonzero weight vector $v
\in V_{\lambda, \mu}$ such that

1) $V$ is generated by $v$ as  ${\mathcal L}$-module;

2) ${\mathcal L}^+ v=0 $ (resp. ${\mathcal L}^- v=0 $), where ${\mathcal L}^+={\mathcal L}_++\c e_0$ and ${\mathcal L}^-={\mathcal L}_-+\c f_0$ (the notations ${\mathcal L}_+=\sum_{i\ge 1}{\mathcal L}_i$, ${\mathcal L}_-=\sum_{i\le -1}{\mathcal L}_i$ are introduced in Section 1).

If, in addition, all weight spaces $V_{\ll, \mu}$ of a weight ${\mathcal L}$-module $V$ are finite-dimensional, the module $V$ is called a {\it
Harish-Chandra module}. Clearly, a highest (lowest) weight module
is a Harish-Chandra module.

For a weight module $V$, we define
\begin{equation}\hbox{Supp}(V):=\bigl\{\lambda\in \c \bigm|V_\lambda=\oplus_{\mu\in\c}V_{\lambda, \mu} \neq
0\bigr\}.\end{equation}

Obviously, if $V$ is an irreducible weight ${\mathcal L}$-module, then
there exists $\lambda\in\c$ such that ${\rm
Supp}(V)\subset\lambda+\z$. So $V=\sum_{i\in \z}V_i$ is a $\z$-graded module, where $V_i=V_{\ll+i}$.

Kaplansky-Santharoubane \cite{KS} in 1983 gave a classification of
Vir-modules of the intermediate series. There are three families
of indecomposable modules with each weight space is one-dimensional:

(1) ${\mathcal A}_{a,\; b}=\sum_{i\in\z}\c v_i$:
$d_mv_i=(a+i+b m)v_{m+i}$;

(2) ${\mathcal A}(a)=\sum_{i\in\z}\c v_i$: $d_mv_i=(i+m)v_{m+i}$
if $i\ne 0$, $d_mv_0=m(m+a)v_{m}$;

(3) ${\mathcal B}(a)=\sum_{i\in\z}\c v_i$:  $d_mv_i=iv_{m+i}$ if
$i\ne -m$, $d_mv_{-m}=-m(m+a)v_0$,  for some $a, b\in\c$, where $C$ acts trivially on the above modules.

It is well-known that ${\mathcal A}_{a,\; b}\cong{\mathcal A}_{a+1,\; b}, \forall a, b\in\c$, then we can always suppose that $a\not\in\z$ or $a=0$ in ${\mathcal A}_{a,\; b}$.
Moreover, the module
${\mathcal A}_{a,\; b}$ is simple if $a\notin\z$ or $b\ne0, 1$.
In the opposite case the module
contains two simple subquotients namely the trivial module and
$\c[t, t^{-1}]/\c$. It is also clear that ${\mathcal A}_{0,0}$ and
${\mathcal B}(a)$ both have $\c v_0$ as a submodule, and their corresponding
quotients are isomorphic, which we denote by ${\mathcal A}_{0,0}'$. Dually,
${\mathcal A}_{0,1}$ and ${\mathcal A}(a)$ both have $\C v_0$ as a quotient module, and
their corresponding submodules are isomorphic to ${\mathcal A}_{0,0}'$. For
convenience, we simply write ${\mathcal A}_{a,b}'={\mathcal A}_{a,b}$ when ${\mathcal A}_{a,b}$ is
irreducible.

All Harish-Chandra modules over the Virasoro algebra were classified in \cite{M} in 1992.
Since then such works were done on the high rank Virasoro algebra in \cite{LvZ0} and \cite{S1}, the Weyl algebra in \cite{S}.

\begin{theo}\cite{M}
Let $V$ be an irreducible weight Vir-module with finite-dimensional weight spaces.
Then $V$ is a highest weight module, lowest weight module, or Harish-Chandra module of intermediate series.
\end{theo}

\begin{theo}\cite{LvZ}\label{thv}
If $V$ is an irreducible weight module with finite-dimensional weight spaces over $\mathcal D$, then $V$ is a highest or lowest weight
module or the Harish-Chandra module of uniformly bounded.
Moreover, any uniformly bounded module is a Harish-Chandra module of intermediate series.
\end{theo}

\noindent{\bf Remarks.}
(1) The Harish-Chandra module of the intermediate series over $\mathcal D$ is induced by ${\mathcal A}_{a,\; b}=\sum_{i\in\z}\c v_i$ with $Y_nv_i=cv_{n+i}$ for some $c\in\c$.
Denote this module by ${\mathcal A}_{a,\; b,\; c}$.
It is well-known that ${\mathcal A}_{a,\; b,\; c}\cong{\mathcal A}_{a+1,\; b,\; c}, \forall a, b, c\in\c$.
Moreover, the module ${\mathcal A}_{a,\; b,\;c}$ is simple if $a\notin\z$ or $b\ne0, 1$ or $c\ne 0$.
We also use ${\mathcal A}_{a,\; b,\; c}'$ to denote by the simple subquotient of ${\mathcal A}_{a,\; b,\; c}$ as in the Virasoro algebra case.

(2) All indecomposable Harish-Chandra modules of
the intermediate series over $\mathcal D$ were classified in \cite{LJ}.

\begin{lemm} \label{lem25}
Let $V$ be a uniformly bounded weight ${\mathcal D}$-module. Then $V$ has an irreducible submodules $V'\cong
{\mathcal A}_{a, b, c}'$ for some $a, b, c\in\c$.
\end{lemm}
\begin{proof}
Consider $V$ as a $\Vir$-module. From
representation theory of $\Vir$ (\cite{KS}), we have $\dim
V_{\ll+n}=p$ for all $\ll+n \neq 0$. We have a $\Vir$-submodule
filtration
    $$0=W^{(0)}\subset W^{(1)} \subset W^{(2)}\subset \cdots \subset W^{(p)}=V,$$
where $W^{(1)}, \cdots ,W^{(p)}$ are $\Vir$-submodules of $V$, and
the quotient modules \break $W^{(i)}/W^{(i-1)}\cong {\mathcal A}_{a_i, b_i}'$ for some $a_i, b_i\in\c$.

Now any ${\mathcal D}$-submodule filtration is also a $\Vir$-submodule filtration and then its length is finite. So $V$ has an irreducible
${\mathcal D}$-submodule $V'$, which is also a uniformly bounded. So by Theorem 2.3, $V'\cong
{\mathcal A}_{a, b, c}'$ for some $a, b, c\in\c$.
\end{proof}

\section{The case of $\dim L=2$}

 Let $T$ be a $2$-dimensional nontrivial Lie algebra. Then we can suppose that $T=\c\{h, e\}$ with $[h, e]=2e$.
 Clearly, there exists an invariant symmetric bilinear form $(\, , )$  on $T$ given by
$$(h, h)=2, \ \ (h, e)=(e, e)=0.$$

In this case, the Lie algebra ${\mathcal T}_2:=T_{av}$ is generated by $\{d_n, e_n, h_n, C\mid n\in\z\}$
\begin{eqnarray*}
&&[d_m,d_n]=(n-m)d_{m+n}+{1\over 12}(n^3-n)\de_{m+n, 0}C,\quad [d_m,e_n]=ne_{m+n},\\
&&[d_m,h_n]=nh_{m+n},\quad
[h_m,e_n]=2e_{m+n},\\
&&[e_m,e_n]=0,\  [h_m,h_n]=2m\de_{m+n, 0}C,
\end{eqnarray*} for all $m,n\in\z$.

Clearly, ${\mathcal T}_2$ is a subalgebra of $\mathcal L$.

\begin{prop} \label{p31}
Let $V$ be a uniformly bounded irreducible
${\mathcal T}_2$-module. Then $V=\sum\c v_i\cong {\mathcal A}_{a,b,c}$ is the Harish-Chandra module of intermediate series with $h_nv_i=cv_{n+i}$ and $e_nv_i=0$.
\end{prop}
\begin{proof}
From representation theory of $\Vir$ (\cite{KS}), we have $C=0$.

Clearly, $\c\{h_0, e_0\}$ is a $2$-dimensional solvable Lie
subalgebra of ${\mathcal T}_2$. So we can choose an irreducible $\c\{h_0,
e_0\}$-submodule $\c\{v\}$ of $V$  such that $h_0v=c_1v$ and
$e_0v=0$ for some $c_1\in\c$.

Now the Lie subalgebra ${\mathcal H}_e:=\c\{d_n, e_n\mid
n\in\z\}$ is isomorphic to the Lie algebra $\mathcal D$ defined in Section 2 (see Definition 2.1).
Set $U=U({\mathcal H}_e)v$, which is ${\mathcal H}_e$-module generated by $v$.
By Lemma \ref{lem25}, we can choose an irreducible ${\mathcal H}_e$-submodule $V'=\sum u_i$ of $U$ with $e_n u_i=du_{i+n}$ for some $d\in\c$ and for all $n, i\in\z$.

Moreover, $V=\sum U(H)u_i$, where $H=\c\{h_i\mid i\in\z\}$. Clearly $e_0$ is nilpotent on some element $u_i\in V'$, so $d=0$.
It is $e_nV'=0$, and then $e_nV=0$ since
$e_nh_iu_j=h_ie_nu_j-2e_{i+j}u_i=0$ for all $n\in\z$. Now the irreducibility of $V$
as ${\mathcal T}_2$-module is equivalent to that of $V$ as ${\mathcal
H}_h$-module, where ${\mathcal H}_h=\c\{d_n, h_n\mid n\in\z\}$ is also isomorphic to the Lie algebra $\mathcal D$. By
Theorem \ref{thv}, $V=\sum v_i$ is the Harish-Chandra module of intermediate series
with $h_nv_i=cv_{n+i}$ for some $c\in\c$ and for all $n, i\in\z$.
\end{proof}

\section{Harish-Chandra modules over the affine-Virasoro algebra $\mathcal L$}

\begin{theo} \label{pH1}
Let $V$ be an irreducible weight ${\mathcal L}$-module with finite-dimensional weight spaces. If $V$ is not a highest and lowest
module, then $V$ is uniformly bounded.
\end{theo}
\begin{proof} From Section 2, we can suppose that $V=\oplus_{i\in \z}V_i$ is an
irreducible Harish-Chandra ${\mathcal L}$-module without highest and
lowest weights. We shall prove that for any $i\in\Z^*$, $k\in\Z$,
\begin{eqnarray}\label{s===0}
d_i|_{V_k}\oplus d_{i+1}|_{V_k}\oplus e_{i}|_{V_k}\oplus
f_{i}|_{V_k}\oplus h_i|_{V_k}: \ \ V_k\ \to\ V_{k+i}\oplus V_{k+i+1}
\end{eqnarray}
is injective. In particular, by taking $i=-k$, we obtain that
$\dim\,V_k$ is uniformly bounded.

In fact, suppose there exists some $v_0\in V_k$ such that
\begin{equation}\label{LLL-1111}d_iv_0=d_{i+1}v_0=e_{i}v_0=f_{i}v_0=h_{i}v_0=0.\end{equation}
Without loss of generality, we can suppose $i>0$. Note that when
$\ell\gg0$, we have
$$\ell=n_1i+n_2(i+1)$$
for some $n_1, n_2\in\N$, from this and the relations in the
definition, one can easily deduce that $d_\ell,e_{\ell}, f_{\ell},
h_{\ell}$ can be generated by $d_i,d_{i+1},e_{i}, f_{i}, h_{i}$.
Therefore there exists some $N>0$ such that
$$d_\ell v_0=e_{\ell}v_0=f_{\ell}v_0=h_{\ell}v_0=0\mbox{\ \ for all \ }\ell\ge
N.$$

This means \begin{equation}{\mathcal L}_{[N, +\infty)}v_0=0,\label{v00}\end{equation} where ${\mathcal L}_{[N, +\infty)}=\oplus_{i\ge N} {\mathcal L}_i$.

Since ${\mathcal L}={\mathcal L}_{[1, N)}+{\mathcal L}_0+{\mathcal L}_-+{\mathcal L}_{[N, \infty)}$, using the PBW theorem and the irreducibility of $V$, we have
\begin{eqnarray}
V&=&U({\mathcal L})v_0=U({\mathcal L}_{[1, N)})U({\mathcal L}_0+{\mathcal L}_-)U({\mathcal L}_{[N, \infty)})v_0\\
&=&U({\mathcal L}_{[1, N)})U({\mathcal L}_0+{\mathcal L}_-)v_0.\label{4.4}
\end{eqnarray}

Note that $V_+$ is a ${\mathcal L}_+$-module. Let $V_+'$  be the ${\mathcal L}_+$-submodule of $V_+$ generated by $V_{[0, N)}$.

Now prove that \begin{equation}V_+=V_+'. \label{vvv}\end{equation} In fact, let $x\in V_+$ be of degree $k$. If $0\le k<N$, then by definition $x\in V_+'$. Suppose that $k\ge N$. By (\ref{4.4}), $x$ is a linear combination of the form $u_ix_i$ with $u_i\in {\mathcal L}_{[1, N)}$ and $x_i\in V$, where $i$ is in a finite subset of $\mathbb {Z}_+$.
For any $i\in I$, the degree $\deg u_i$ of $u_i$ satisfies $1\le \deg u_i<N$, so $0<\deg x_i=k-\deg u_i<k$. By inductive hypothesis, $x_i\in V_+'$, and thus $x\in V_+'$. So (\ref{vvv}) holds.

Eq. (\ref{vvv}) means that $V_+$ is finite generated as ${\mathcal L}$-module. Choose a basis $B$ of $V_{[0, N)}$, then for any $x\in B$, we have $x=u_xv_0$ for some $u_x\in U(\mathcal L)$. Regarding $u_x$ as a polynomial with respect to a basis of $\mathcal L$, by induction on the polynomial degree and using $[u, w_1w_2]=[u, w_1]w_2+w_1[u, w_2]$ for $u\in \mathcal L$, $w_1, w_2\in U(\mathcal L)$, we see that there exists a positive integer $k_x$ large enough such that $k_x>N$ and $[{\mathcal L}_{[k_x, +\infty)}, u_x]\subset U(\mathcal L){\mathcal L}_{[N, \infty)}$.

Then by (\ref{v00}), ${\mathcal L}_{[k_x, +\infty)}x=[{\mathcal L}_{[k_x, +\infty)}, u_x]v_0+u_x{\mathcal L}_{[k_x, +\infty)}v_0=0$. Take $k=\max\{k_x, x\in B\}$, then
$${\mathcal L}_{[k, +\infty)}V_+={\mathcal L}_{[k, +\infty)}U({\mathcal L}_+)V_{[0, N)}=U({\mathcal L}_+){\mathcal L}_{[k, +\infty)}V_{[0, N)}=0.$$
Since ${\mathcal L}_+\subset {\mathcal L}_{[k, +\infty)}+[{\mathcal L}_{[-k', 0)}, {\mathcal L}_{[k, +\infty)}]$ for some $k'>k$, we get ${\mathcal L}_+V_{[k', +\infty)}=0$. Now if $x\in V_{[k'+N, +\infty)}$, by (\ref{4.4}), it is a sum of elements of the form $u_jx_j$ such that $u_j\in {\mathcal L}_{[1, +\infty)}$ and then $x_j\in V_{[k', +\infty)}$, and thus $u_jx_j=0$. This prove that $V$ has no degree $\ge k'+N$.

Now let $p$ be the maximal integer such that $V_p\ne 0$, since the four-dimensional subalgebra $\c\{d_0, h_0, e_0, C\}$ has a two-dimensional solvable subalgebra $\c\{h_0, e_0\}$ and two central elements $\{d_0, C\}$, so there exists a common eigenvector $w$ of $\frak h=\c\{d_0, h_0, C\}$ in $V_p$ with $e_0w=0$. It is ${\mathcal L}^+w=0$.
Then $w$ is a highest weight vector of $\mathcal L$, this contradicts the assumption of the Theorem.
\end{proof}

\section{Representations of the Lie algebra $\mathcal L$}

Now we shall consider uniformly bounded irreducible
 weight modules over ${\mathcal L}$.

Let $M(\ll)$ be the finite-dimensional irreducible highest weight
$\frak{sl}_2$-module with highest weight $\ll$, then
$L(M(\ll)):=M(\ll)\ot \c[t, t^{-1}]$ becomes an irreducible
${\mathcal L}$-module by the actions as follows:
\begin{eqnarray*}
&&d_m(u\ot t^i)=(a+bm+i)u\ot t^{m+i},\\
&&x_m(u\ot t^i)=(x\cdot u)\ot t^{m+i},
\end{eqnarray*} for any $u\in M(\ll)$ and for some $a, b\in\c$.

\noindent{\bf Remark.} $L(M(\ll))$ irreducible iff $M(\ll)$ is a nontrivial $\frak{sl}_2$-module or $a\not\in\z$ or $b\ne 0, 1$. We also use $L(M(\ll))'$ to denote the irreducible submodule or subquotient of $L(M(\ll))$.

\begin{prop} \label{p51}
Let $V$ be a  uniformly bounded irreducible
 weight ${\mathcal L}$-module. Then $V$ is isomorphic to $L(M(\ll))'$ for some finite-dimensional irreducible ${\mathcal L}$-module $M(\ll)$.
\end{prop}
\begin{proof}
Similarly to Proposition \ref{p31}, one has $C=0$. Clearly, ${\mathcal T}_2=\c\{d_n, h_n, e_n, C\mid n\in\z\}$
is a subalgebra of ${\mathcal L}$.

Consider $V$ as a ${\mathcal
T}_2$-module, similarly to Lemma \ref{lem25}, $V$ has an irreducible uniformly bounded submodule $V'$. By Proposition \ref{p31},
$V'=\sum\c v_i$ of $V$ with $h_nv_i=cv_{n+i}$ and
$e_{n}v_i=0$ for some $c\in\c$. Moreover, $V=U(F)V'$, where $F=\c\{f_i\mid i\in\z\}$ is the Lie subalgebra of $\mathcal L$ generated by $f_i$ for all $i\in\z$.

If $c=0$, then $e_if_jv_k=[e_i, f_j]v_k+f_je_iv_k=0$ for all $u, j, k\in\z$. So $EV=0$, where $E=\c\{e_i\mid i\in\z\}$ is the Lie subalgebra of $\mathcal L$ generated by $e_i$ for all $i\in\z$.

Then the irreducibility of $V$
as ${\mathcal L}$-module is equivalent to that of $V$ as ${\mathcal
T}_2'$-module, where the subalgebra ${\mathcal
T}_2':=\c\{d_n, h_n, f_n, C\mid n\in\z\}$ is also isomorphic to the Lie algebra ${\mathcal T}_2$. By
Proposition \ref{p31}, $V$ is the Harish-Chandra module of intermediate series
with $f_nv_i=0$ for all $n, i\in\z$. The theorem is proved. So we can suppose that $c\ne 0$.

Fix $k\in\z$, for any $i\in\z$, $v_i={\frac1c}h_{i-k}v_k$, so $U(F)v_i\subset U(H, F)v_k$, where $U(H, F)$ is the universal enveloping algebra of the Lie subalgebra generated by $h_i, f_i$ for all $i\in\z$ of $\mathcal L$.
So $$V=U(H, F)v_k, \eqno(5.1)$$ for any $k\in\z$.
Then for any $v\in V$ there exists $n\in\z_+$ such that $e_i^nv=0$ since $e_iv_k=0$. It is that each $e_i$ is locally nilpotent on $V$.

Replacing ${\mathcal T}_2=\c\{d_n, h_n, e_n, C\mid n\in\z\}$ by the subalgebra ${\mathcal T}_2':=\c\{d_n, h_n, f_n, C\mid n\in\z\}$, we see that each $f_i$ is locally nilpotent on $V$.
So $V$ is an integrable weight ${\mathcal L}$-module.

By (5.1) we know that $V$ becomes an irreducible module over the loop algebra $\bar L({\mathcal L})=\c\{e_i, f_i, h_i, d_0\mid i\in\z\}$. Moreover, $V$ is an integrable weight $\bar L({\mathcal L})$-module. So by \cite{C} (or see \cite{CP}, \cite{E0}), $V\cong M(\bar\ll, \bar a)=L(\otimes_{i=1}^kM(\ll_i))$ as $\bar L({\mathcal L})$-modules, where $\bar \ll=(\ll_1, \cdots, \ll_k)$ and $\bar a=(a_1, \cdots, a_k)$.
However, by $[d_i, h_j]=jh_{i+j}$, we have $k=1$ and $\bar a=1$. So as an irreducible $\bar L({\mathcal L})$-module, $V\cong L(M(\ll))'$ for some highest weight $\ll$ of $\frak{sl}_2$.

Now we only need consider the actions of $d_n$ on $V$. Suppose that $\sum_{i\in\z}\c v\ot t^i$ is a Vir-module of intermediate series, then
$$d_nf_0v\ot t^i=f_0d_nv\ot t^i=(a+bn+i)f_0v\ot t^{n+i}.$$
So $\sum_{i\in\z}\c f_0v\ot t^i$ is a Vir-module of intermediate series.
\end{proof}

Combining with Theorem 4.1, we get the following result.

\begin{theo} \label{main2}
Let $V$ be an irreducible weight ${\mathcal L}$-module with finite-dimensional weight spaces. Then $V$ is a highest weight module or a
lowest weight module or isomorphic to $L(M(\ll))'$ for some finite-dimensional irreducible ${\mathcal L}$-module $M(\ll)$.
\end{theo}

\noindent{\bf Remark.}
The unitary highest weight modules over the affine-Virasoro algebra $\mathcal L$  were considered in \cite{Ka1, JY}.

\newpage \centerline{\bf ACKNOWLEDGMENTS}

\vskip15pt
Project is supported by the NNSF (Grants: 11271131, 11371134) and ZJNSF (LZ14A010001).


\begin{thebibliography}{15}

\bibitem{ADKP} E. Arbarello; C. De Concini; V.G. Kac; C. Procesi,
{\sl Moduli spaces of curves and representation theory}, Comm. Math. Phys., {\bf 117} (1988), 1--36.


\bibitem{B}Y. Billig, A category of modules for the full toroidal Lie algebra, International Mathematics Research Notices, 2006, 46 pp



\bibitem{C} V. Chari, {\sl Integrable representations of affine Lie-algebras}, Invent. Math. {\bf 85} (2) (1986), 317--335.

\bibitem{CP} V. Chari; A. Pressley, {\sl New unitary representations of loop groups}, Math. Ann. {\bf 275} (1986), 87--104.

\bibitem{CL} S. Cheng; N. Lam, {\sl Finite conformal modules over the $N = 2, 3, 4$ superconformal
algebras}. J. Math. Phys. {\bf 42} (2) (2001), 906--933.

\bibitem{E0} S. Eswara Rao, {\sl On the representations of loop algebras}, Comm. Algebra {\bf 21} (1993), 2131--2153.

\bibitem{E1} S. Eswara Rao, {\sl Classification of irreducible integrable modules for multi-loop algebras with finite-dimensional weight spaces}, J. Algebra {\bf 246} (2001),  215--225.

\bibitem{EJ} S. Eswara Rao; C. Jiang, {\sl Classification of irreducible integrable representations for the full toroidal Lie algebras}, J. Pure Appl. Algebra {\bf 200} (1-2) (2005), 71--85.

\bibitem{FK} I.B. Frenkel; V.G. Kac, {\sl Basic representations of affine Lie algebras and dual resonance
models}, Invent. Math. {\bf 62} (1980), 23--66.

\bibitem{GLZ} X. Guo; R. Lv; K. Zhao, {\sl Simple Harish-Chandra modules, intermediate series modules, and Verma modules over the loop-Virasoro algebra}, Forum Math. {\bf 23} (2011), 1029--1052.

\bibitem{JM} C. Jiang; D. Meng, {\sl Integrable representations for generalized Virasoro-toroidal Lie algebras}, J. Algebra {\bf 270} (1) (2003), 307--334.

\bibitem{JY} C. Jiang; H. You, {\sl Irreducible representations for the affine-Virasoro Lie algebra of type $B_l$}, Chinese Ann. Math. Ser. B {\bf 25} (3) (2004), 359--368.

\bibitem{Ka1} V. G. Kac, {\sl Highest weight representations of conformal current algebras}, Symposium on Topological and Geometric Methods in Field Theory, Espoo, Finland, World Scientific, p. 3--16, 1986.

\bibitem{Ka} V. G. Kac, {\sl Infinite-dimensional Lie Algebras}, 3rd ed., Cambridge Univ. Press, Cambridge,
U.K., 1990.

\bibitem{KS} I. Kaplansky; L. J. Santharoubane, {\sl Harish-Chandra modules
over the Virasoro algebra}, Infinite-dimensional groups with
applications (Berkeley, Calif. 1984), 217--231, Math. Sci. Res.
Inst. Publ., 4, Springer, New York, 1985.


\bibitem{Ku} G. Kuroki, {\sl Fock space representations of affine Lie algebras and integral representations in the Wess-Zumino-Witten models}, Comm. Math. Phys. {\bf 142} (3) (1991), 511--542.

\bibitem{LH} D. Liu; N. Hu, {\sl Vertex representations for toroidal Lie algebra of type $G_2$}, J. Pure Appl. Algebra {\bf 198} (1-3) (2005), 257--279.

\bibitem{LJ} D. Liu; C. Jiang, {\sl Harish-Chandra modules over the twisted Heisenberg-Virasoro algebra},
J. Math. Phys., {\bf 49} (1) (2008), 012901, 13 pp.

\bibitem{LQ} X. Liu; M. Qian, {\sl Bosonic Fock representations of the affine-Virasoro algebra}, J. Phys. A, {\bf 27} (5) (1994), 131--136.

\bibitem{LvZ0} R. Lv;  K. Zhao, {\sl Classification of irreducible
weight modules over higher rank Virasoro algebras},  Adv. Math. {\bf 201} (2) (2006), 630--656.

\bibitem{LvZ} R. Lv; K. Zhao, {\sl Classification of irreducible weight modules over the twisted
 Heisenberg-Virasoro algebra}, Comm. Contemp. Math. {\bf 12} (2) (2010), 183--205.

\bibitem{M} O. Mathieu,  {\sl Classification of Harish-Chandra
modules over the Virasoro Lie algebra}, Invent. Math. {\bf
107} (1992), 225--234.

\bibitem{MRY} R. V. Moody; S. Eswara Rao; T. Yomonuma, {\sl Toroidal Lie algebras and vertex representations},
Geom. Ded. {\bf 35} (1990), 287--307.

\bibitem{RM} S. Eswara Rao; R. V. Moody, {\sl Vertex represenations for $n$-affine-Virasoro  Lie
algebras and a generalization of the Virasoro algebra}, Comm. Math.
Phys. {\bf 159} (1994), 239--264.

\bibitem{S} Y. Su, {\sl Classification of quasifinite modules over the Lie algebras of
Weyl type}, Adv. Math. {\bf 174} (2003), 57--68.


\bibitem{S1} Y. Su, {\sl Classification of Harish-Chandra modules over the
higher rank Virasoro algebras}, Comm. Math. Phys. {\bf 240}
(2003), 539--551.

\bibitem{W} M. Wakimoto, Lectures on infinite-dimensional Lie Algebra, World Scientific Publishing Co., Inc. 2001.

\bibitem{XH} L. Xia; N. Hu, {\sl Irreducible representations for Virasoro-toroidal Lie algebras}. J. Pure Appl. Algebra {\bf 194} (1-2) (2004), 213--237.
\end{thebibliography}
\end{document}